\documentclass{amsart}
\usepackage{amsfonts}
\usepackage{amsmath,amscd}
\usepackage{amssymb}
\usepackage{amsthm}
\usepackage{newlfont}
\newcommand{\f}{\frac}

\newcommand{\ds}{\displaystyle}

 \newtheorem{thm}{Theorem}[section]
 \newtheorem{cor}[thm]{Corollary}
 \newtheorem{lem}[thm]{Lemma}
 \newtheorem{prop}[thm]{Proposition}
 \theoremstyle{definition}
 
 \theoremstyle{remark}

 \numberwithin{equation}{section}

\begin{document}

\title[Some properties on the tensor square of Lie algebras]
 {Some properties on the tensor square of Lie algebras}

\author[P. Niroomand]{Peyman Niroomand}
\address{School of Mathematics and Computer Science\\
Damghan University, Damghan, Iran}
\email{p$\_$niroomand@yahoo.com, niroomand@du.ac.ir}




\date{\today}

\begin{abstract}In the present paper we extend and improve the results of \cite{bl, br} for the tensor square of Lie algebras.
More precisely,  for any Lie algebra $L$ with $L/L^2$ of finite
dimension, we prove $L\otimes L\cong L\square L\oplus L\wedge L$ and
$Z^{\wedge}(L)\cap L^2=Z^{\otimes}(L)$. Moreover, we show that $L\wedge L$ is isomorphic to derived subalgebra of a cover of $L$, and finally we give a free presentation for it.
\end{abstract}

\maketitle

\section{Introduction and preliminaries}

Let $L$ and $H$ be Lie algebras over a fixed field  $F$ and let $[,]$ denotes
the Lie bracket. A bilinear map $\rho: L\times L\rightarrow H$ is said to be a Lie pairing provided that
\begin{itemize}
\item[(i)]$\rho([l,l' ], s) = \rho(l, [l' , s]) - \rho(l', [l, s]),$
\item[(ii)]$\rho(l, [s, s']) = \rho([s', l], s) - \rho([s, l], s'),$
\item[(iii)]$\rho([l, s], [l', s']) = -[\rho(s, l), \rho(l', s')],$
for all $l,l',s,s' \in L$.
\end{itemize}
For instance it can be checked that the map $L\times L\rightarrow L^2$ is a Lie pairing.

The theory of tensor square of groups was extended to tensor square of Lie algebras  by Ellis in \cite{el}.

Specifically, the tensor square $L\otimes L$ is a Lie algebra generated by the symbols $l\otimes k$ subject to the following relations
\begin{itemize}
\item[(i)] $c(l\otimes k) = cl\otimes k = l\otimes ck,$
\item[(ii)] $(l+l')\otimes k = l\otimes k + l'\otimes k,$\\
$l\otimes(k + k') = l\otimes k + l\otimes k',$
\item[(iii)] $[l, l']\otimes k = l\otimes[l',k]-l'\otimes[l,k],$\\
$l\otimes[k, k'] = ([k',l])\otimes k - [k,l]\otimes k',$
\item[(iv)] $[(l\otimes k), (l'\otimes k')] = -[k,l] \otimes[k',l']$ for all $c\in F$, $l, l', k, k'\in L$.
\end{itemize}

Recall from \cite[Proposition 1]{el} the mapping
\[h:L\times L\rightarrow L\otimes L, (l,l')\mapsto l\otimes l'\]
is a universal Lie pairing, that is for any Lie pairing
$h' :L \times L\rightarrow Q$ there is a unique Lie homomorphism $\zeta: L\otimes L\rightarrow Q$ such that $\zeta h=h'$.

Let $L\square L$ denotes the submodule of $L\otimes L$ generated by $l\otimes l$ for all $l\in L$.
Then the exterior square $L\wedge L$ of $L$ is the quotient $L\otimes L/L\square L$, for $l\otimes l'\in L\otimes L$ we denote the coset $l\otimes l'+ L\square L$ by $l\wedge l'$.

Ellis in \cite{el} defined $J_2(L) = \mathrm{ker}(L\otimes
L\stackrel{\kappa}\rightarrow L^2, l \otimes l'\mapsto [l, l'])$.
The kernel of $\kappa$ contains $L\square L$, thus $\kappa$ induces
an epimorphism $\kappa ':L\wedge L\rightarrow L^2$. It is known that
the Schur multiplier $\mathcal{M}(L)$ is the kernel of $\kappa '$ by the
result of \cite{el1}.

The concept of  Whitehead's quadratic functor $\Gamma $ $($as defined in \cite{wh} for groups and \cite{el} for Lie algebras$)$ gives
 the natural epimorphism $\Gamma(L/L^2)\rightarrow L\square L$ defined by $\gamma(L^2+l)\mapsto l\otimes l$ and
helps to survey the connection between $L\square L$ and $L/L^2\square L/L^2$.

Several authors tried to extend known results of groups to Lie
algebras $($see for instance \cite{ba, bo, el1, el, ma, ni3, ni, ni1,ni2, sa1, sa}$).$

 The present paper is devoted  to  develop and improve an analogous theory of tensor
products from \cite{bl, br} for Lie algebras.
More precisely,  we obtain strong Lie algebra version of \cite{bl} that is
\[L\otimes L\cong L\square L\oplus L\wedge L~\text{and}~Z^{\wedge}(L)\cap L^2=Z^{\otimes}(L)\] where $L/L^2$ is a finite dimensional Lie algebra and  the tensor and exterior centers of a Lie algebra $L$ are defined as follows
\begin{center}
  \begin{tabular}{lp{7cm}}
$Z^{\otimes}(L)=$&$\{~l\in L~|~l\otimes l'=0,~\text{for all}~l'\in L\}$ and \\
$Z^{\wedge}(L)=$&$\{~l\in L~|~l\wedge l'=0,~\text{for all}~l'\in L\},$
\end{tabular}
  \end{center}
respectively. Moreover, we prove that $L\wedge L$ is isomorphic to derived subalgebra of a cover of $L$, and we give a free presentation for the exterior square  analogues to \cite{br, bl}.

The notation and terminology which are used throughout this paper is standard and follows  \cite{el1, el, ni}.

\section{Main Results}
This section is devoted  to develop and improve a  Lie algebra version of  \cite{bl, br}.

First we state the following two lemmas which will be used in the rest.
\begin{lem}\label{l1e}$($See \cite[Corollary 3.2]{ni}$)$ Let $L$ be a finite dimensional Lie algebra and $\pi:L\otimes L\rightarrow L/L^2\otimes L/L^2$ be the natural epimorphism. Then the restriction  $\pi:L\square L\rightarrow L/L^2\square L/L^2$  is an isomorphism.
\end{lem}
Lemma \ref{l1e} can be improved as follows
\begin{lem}\label{l2e} Let  $L$ be a Lie algebra such that $L/L^2$ is of finite dimension. Then the restriction  $\pi:L\square L\rightarrow L/L^2\square L/L^2$ is an isomorphism.
\end{lem}
\begin{proof}The natural epimorphism $\Gamma(L/L^2)\rightarrow L\square L$ guarantees $L\square L$ to be of finite dimension, and hence
 the restriction $\pi:L\square L\rightarrow L/L^2\square L/L^2$ is an isomorphism by Lemma \ref{l1e}.
\end{proof}

\begin{lem}\label{l1} Let $L$ be an abelian Lie algebra with a basis  $\{x_1,x_2,\ldots,x_n\}$. Then \[
L\otimes L\cong L\square L\oplus\langle x_i\otimes x_j~|~1\leq i<j\leq n \rangle.\]
\end{lem}
\begin{proof}
Since $L\cong \bigoplus_{i=1}^n\langle x_i\rangle$, we have \[L\otimes L\cong\langle x_i\otimes x_j+x_j\otimes x_i, x_i\otimes x_i~|~1\leq i<j\leq n \rangle\oplus\langle x_i\otimes x_j~|~1\leq i<j\leq n \rangle.\] The result follows from the fact $L\square L\cong \langle x_i\otimes x_j+x_j\otimes x_i, x_i\otimes x_i~|~1\leq i<j\leq n \rangle.$
\end{proof}
\begin{lem}\label{l2}Let $L$ be a Lie algebra and let $\pi :L\otimes L\rightarrow L/L^2\otimes L/L^2$ be the natural epimorphism. Then $\mathrm{Ker}~\pi=L\otimes L^2+L^2\otimes L=L\otimes L^2 $.
\end{lem}
\begin{proof} Taking $N=L\otimes L^2+L^2\otimes L$. The epimorphism $\pi$ has $N$ in its kernel and $N$ is an ideal of $L\otimes L$ by invoking (iv), therefore inducing a map \[\bar\pi:(L\otimes L)/N\rightarrow L/L^2\otimes L/L^2.\] On the other hand,
let $\bar x$ denotes $x$ modulo $L^2$. Define \[\alpha:L/L^2\times L/L^2\rightarrow (L\otimes L)/N \text{ by}~ (\bar x,\bar y)\mapsto N+x\otimes y.\] It readily shown that $\alpha$ is well-defined. Let $l,l',k,k'\in L$, then
 \[\begin{array}{lcl}\alpha([\bar l,\bar{l'}],k)&=&N+[l,l']\otimes k=N+(l\otimes [l',k]-l'\otimes[l,k])\vspace{.3cm}\\&=&N+l\otimes [l',k]-N+l'\otimes[l,k]=\alpha(\bar l,[\bar{l'},\bar k])-\alpha(\bar{l'},[\bar l,\bar k]).\end{array}\] By a similar fashion, it is easily seen that
 \[\alpha(\bar l, [\bar{k}, \bar{k'}]) = \alpha([\bar {k'},l], k) - h([\bar{k},\bar{l}], \bar{k'})~\text{and}~
\alpha([\bar{k},\bar{l}],[\bar{l'},\bar{k'}]) =
-[\alpha(l,k),\alpha(\bar{l'},\bar{k'})].\] Thus $\alpha$ is a Lie
pairing, and hence  induces a Lie homomorphism
$\bar\alpha:L/L^2\otimes L/L^2\rightarrow (L\otimes L)/N$ by using
\cite[Lemma 1.1]{sa}. Hence $(L\otimes L)/N\cong L/L^2\otimes L/L^2$
since $\bar\pi\bar\alpha$ and $\bar\alpha\bar\pi$ are identity. The equality $L\otimes L^2+L^2\otimes L=L\otimes L^2$ is obtained directly from relations of $L\otimes L$.
\end{proof}
\begin{thm}\label{t2} Let $L/L^2$ be a finite dimensional Lie algebra. Then \[L\otimes L\cong L\wedge L\oplus L\square L.\]
\end{thm}
\begin{proof}  By invoking Lemma \ref{l1}, we have \[L/L^2\otimes L/L^2\cong L/L^2\square L/L^2\oplus\langle \bar{x_i}\otimes \bar{x_j}~|~1\leq i<j\leq n\rangle .\]
Now consider the map $\pi:L\otimes L\rightarrow L/L^2\otimes L/L^2$,
we should have \[L/L^2\square L/L^2\oplus\langle \bar{x_i}\otimes
\bar{x_j}~|~1\leq i<j\leq n \rangle =\pi(L\square L+\langle
x_i\otimes x_j~|~1\leq i<j\leq n\rangle),\] where $\{\bar x_1,\bar
x_2,\ldots,\bar x_n\}$ is a basis for $L/L^2$, and hence $L\otimes
L=L\square L+\langle x_i\otimes x_j~|~1\leq i<j\leq n\rangle)+N$ due
to Lemma \ref{l2}.
 By virtue of Lemma \ref{l2e}, the restriction of $\pi$ maps $x$ to zero for all $x\in L\square L\cap(\langle x_i\otimes x_j~|~1\leq i<j\leq n\rangle)+N)$, which implies that
 \[L\otimes L\cong L\square L\oplus(\langle x_i\otimes x_j~|~1\leq i<j\leq n\rangle)+N)\] via Lemma \ref{l2e}.
 It is clear that $L\wedge L\cong \langle x_i\otimes x_j~|~1\leq i<j\leq n\rangle+N$, as required.
 \end{proof}
\begin{cor}Let $L/L^2$ be a finite dimensional Lie algebra. Then \[J_2(L)\cong L\square L\oplus \mathcal{M}(L).\]
\end{cor}
\begin{proof} By Theorem \ref{t2}, \[\begin{array}{lcl}J_2(L)&=&\big(L\square L\oplus(\langle x_i\otimes x_j~|~1\leq i<j\leq n\rangle+N)\big)\cap J_2(L)\vspace{.3cm}\\&=&L\square L\oplus\big((\langle x_i\otimes x_j~|~1\leq i<j\leq n\rangle+N)\cap J_2(L)\big).\end{array}\]
The rest of proof is obtained by the fact $J_2(L)/L\square L\cong \mathcal{M}(L)$.
\end{proof}
\begin{cor} Let $L/L^2$ be a finite dimensional Lie algebra. Then $Z^{\wedge}(L)\cap L^2=Z^{\otimes}(L)$.
\end{cor}
\begin{proof} Let $l\in Z^{\wedge}(L)\cap L^2$. Then for all $l'\in L$ we have $l\otimes l'\in L\square L$, and also $l\otimes l'\in\mathrm{Ker}\pi$. Thus Lemma \ref{l2e} implies that $l\otimes l'=0$, and hence $l\in Z^{\otimes}(L).$  Conversely $Z^{\otimes}(L)\subseteq L^2$ since for all $l\in Z^{\otimes}(L)$, we have $l+L^2\otimes l'+L^2=0$ which implies that $l\in L^2$. The result follows.
\end{proof}
A pair of Lie algebras $(K,M)$ is called a defining pair
for $L$ if
\begin{itemize}
\item[(i)]$0\rightarrow M \rightarrow K \rightarrow L \rightarrow 0$ is exact;
 \item[(ii)]$M \subseteq Z(K) \cap K^2$.
\end{itemize}  When $L$ is finite
dimensional then the dimension of  $K$ is bounded. The $K$ of
maximal dimension is called a cover of $L$, and the
corresponding $M$, is the Schur multiplier of L.
From \cite{ba}, the Schur multiplier of $L$ can be also defined in term of  free Lie algebra. More precisely, for an exact sequence
 \[0\rightarrow R \rightarrow F \rightarrow L \rightarrow 0,\] where $F$ is a free Lie algebra, the Schur multiplier of
$L$ is isomorphic to the factor Lie algebra $R\cap F^2/[R,F]$. It is known by \cite{ma} that for a finite dimension Lie algebra Lie covers always exist and are unique up to isomorphism.

We are interested in proving that the derived subalgebra of a covering Lie algebra of $L$ is isomorphic to $L\wedge L$. These developes the results which have been obtained in \cite{bl, br}.
First we recall the following proposition from \cite{el1}.
 \begin{prop}\label{np} Let $F$ be a free Lie algebra. Then  \[F^2\rightarrow F\wedge F, [x,y]\mapsto x\wedge y\] is an isomorphism.
 \end{prop}

\begin{thm}Let $L$ be a Lie algebra with a covering Lie algebra $\bar L$. Then
\[L\wedge L\cong \bar L^2.\]
\end{thm}

\begin{proof} Let $\bar L$ be a cover of $L$. Then there is an exact sequence
 \[0\rightarrow \mathcal{M}(L) \rightarrow \bar L \rightarrow L \rightarrow 0.\] Suppose that $\bar L\cong F/R$  where $F$ is a free Lie algebra, so
$L$ is isomorphic to $F/S$ in which $\mathcal{M}(L)\cong S/R$.

By invoking Proposition \ref{np}, there is a  Lie homomorphism $$\eta:F^2\rightarrow F/S\wedge F/S$$ sending
${\alpha_1}[x_1,y_1]+\ldots+{\alpha_k}[x_k,y_k] $ to ${\alpha_1}(S+x_1\wedge S+y_1)+\ldots+ {\alpha_k}(S+x_k\wedge S+y_k)$.
Evidently, $\eta$ factors through $R\cap F^2$, to induces a homomorphism \[\bar\eta: \f{R\cap F^2}{R} \rightarrow F/S\wedge F/S.\]
From the definition of $\bar\eta$, its surjectivity follows immediately. We claim that $\bar\eta$ is injective.

Define $\ds\theta:F/S\times F/S\rightarrow \f{F^2}{R\cap F^2} $ by $\theta(S+f,S+f_1)=R\cap F^2+[f,f_1]$. Since $[S,F]\subseteq R\cap F^2$, one can easily check that $\theta$ is well-defined.
Since $h:F/S\times F/S\rightarrow F/S\otimes F/S$ is a universal Lie pairing, there exists a Lie homomorphism $\ds\tau:F/S\otimes F/S\rightarrow \f{F^2}{R\cap F^2} $ such that $\tau h=\theta$.
Of course, $\tau$ is trivial on the ideal $F/S\square F/S$, and so it
induces a Lie homomorphism $$\bar\tau:F/S\wedge F/S\rightarrow \ds\f{F^2}{R\cap F^2} .$$ Now both of $\bar\tau\bar\eta$ and $\bar\eta\bar\tau$ are identity, and it implies that $$F/S\wedge F/S\cong \ds\f{F^2}{R\cap F^2}\cong \f{F^2+R}{R},$$ as required.
\end{proof}
\begin{thm} Let $0\rightarrow R \rightarrow F \rightarrow L \rightarrow 0$ be a free presentation of a Lie algebra $L$. Then \[L\wedge L\cong F^2/[R,F].\]
\end{thm}
\begin{proof}
Using Proposition \ref{np}, there is a  Lie homomorphism $\varepsilon:F^2\rightarrow F/R\wedge F/R$ sending
${\alpha_1}[x_1,y_1]+\ldots+{\alpha_k}[x_k,y_k] $ to ${\alpha_1}(R+x_1\wedge R+y_1)+\ldots+{\alpha_k}(R+x_k\wedge R+y_k)$.
Since $\varepsilon$ sends $[R,F]$ to identity, it induces a Lie homomorphism $\bar\varepsilon:F^2/[R,F]\rightarrow  F/R\wedge F/R$.
On the other hand, since $\xi:F/R\times F/R\rightarrow F^2/[R,F]$ is a Lie pairing, obviously there exists a Lie homomorphism
$\bar\xi:F/R\wedge F/R\rightarrow F^2/[R,F]$ sending $R+f\wedge R+f_1$ to $[R,F][f,f_1]$. It is readily obtained that $\bar\varepsilon\bar\xi$ and
$\bar\xi\bar\varepsilon$ are identity, and the proof is complete.
\end{proof}

\end{document}